\documentclass{article}
\usepackage{graphicx} 
\usepackage{amsmath}
\usepackage{amsthm}
\usepackage{amssymb}
\usepackage{mathtools}
\usepackage[english]{babel}
\usepackage{mymacros}

\usepackage[margin=1.5in]{geometry}
\usepackage{color}
\usepackage{hyperref}
\usepackage{bm}
\usepackage{graphicx}
\numberwithin{equation}{section}

\title{Cyclicity via {weak$^\ast$ sequential cyclicity} in Radially weighted Besov spaces}
\author{Anusrika Datta and Stefan Richter}
\date{\today}

\newcommand{\dB}{\partial \mathbb B_d}

\newcommand{\MH}{\mathrm{Mult}( H)}

\newcommand{\la}{\langle}
\newcommand{\ra}{\rangle}
\newcommand{\Bd}{\mathbb B_d}

\begin{document}

\maketitle

\begin{abstract}
A radially weighted Besov space $H$ is a space of holomorphic functions on the unit ball
$\mathbb{B}_d \subseteq  \mathbb{C}^d$ whose $N$-th radial derivative is square integrable with
respect to a given admissible radial measure. We write $Mult(H)$ for its multiplier algebra.
The cyclic vectors in $H$ are those functions $f$ whose multiplier multiples are dense in $H$.

We call a multiplier
$f \in Mult(H)$  \emph{weak$^\ast$ sequentially cyclic} if its multiplier multiples
are weak$^\ast$ sequentially dense in $Mult(H)$. It is immediate that every weak$^\ast$
sequentially cyclic multiplier is cyclic, and it turns out that the two notions coincide whenever $H$
has the complete Pick property. However, in more general radially weighted Besov spaces there may be multipliers that are cyclic, but not weak$^\ast$ sequentially cyclic.

For bounded holomorphic functions $f$ with no zeros in $\mathbb{B}_d$, we obtain a condition
on $\log f$ that implies the cyclicity of $f$ in $H$ and yields invertibility properties for
$1/f$ within an associated Smirnov-type class. This condition is formulated in terms of
weak$^\ast$ sequentially cyclic multipliers and can often be verified using a comparison
principle: if $f, g \in Mult(H)$ satisfy $|f| \leq |g|$ and if $f$ is weak$^\ast$ sequentially
cyclic, then $g$ is also weak$^\ast$ sequentially cyclic.

These results provide new insights into cyclicity phenomena in radially weighted Besov
spaces in settings where $H$ fails to be a complete Pick space.
\end{abstract}


\noindent \textbf{Subjclass:}
[2020]{Primary 47B32, 47A16; Secondary 32A37, 32H15, 30H25.}

\section{Introduction}

Cyclicity is a fundamental concept in the study of invariant subspaces of analytic function spaces and has deep connections with operator theory, factorization theory, and the function-theoretic properties of multipliers.

Let $H$ be a Hilbert space of holomorphic functions on the unit ball $\mathbb{B}_d$, and let $\mathrm{Mult}(H)=\{\phi:\Bd \to \C; \phi f \in H \text{ for all }f\in H\}$ be its multiplier algebra. We have provided  definitions of all  required concepts in Section 2 and 3. For now just recall that a function $f \in H$ is said to be \emph{cyclic}, if the smallest closed multiplier invariant subspace generated by $f$  coincides with $H$. The problem of characterizing cyclic vectors has been extensively studied in functional analysis.  In the one-dimensional setting of the Hardy space of the unit disc, cyclic functions are precisely the outer functions,  \cite{Beurling1948}. Extending this type of characterization to other spaces, such as the Bergman space and the Dirichlet space, has been an active area of research; see e.g. \cite{HeKoZhuBook}, \cite{ElFaKeMaRaBook}. In these settings, cyclicity is related to the zero sets, boundary behavior, and the structure of invariant subspaces. In several variables, fewer results are available.  If $d=2$, then in \cite{Kosinskivavitsas2023} the authors  determined  necessary and sufficient conditions for polynomials to be cyclic in weighted Dirichlet type spaces of the unit ball. These results have been generalized to certain to smooth functions and higher dimensions in  \cite{Ziarati2025}. For some general cyclicity conditions for  the Drury--Arveson space $H^2_d$ and more general weighted Besov spaces that have the so-called Pick property, we refer the reader to \cite{AlemanPerfektRichterSundbergSunkes2023} and \cite{Richteriteratedlog}. In this paper we will say that $H$ is a complete Pick space, if there is a Hilbert space norm on $H$ such that the corresponding reproducing kernel is a normalized complete Pick kernel, i.e. $k_w(z)= \frac{1}{1-\la u(z), u(w)\ra}$ for some function $u$ from $\Bd$ into the unit ball of $\ell^2$.  One of the advantages of Pick spaces is that frequently questions about functions in the space can be reduced to questions about functions in the multiplier algebra. In this paper our work will directly focus on the multiplier algebra, and we will establish results for spaces that don't necessarily have the Pick property. Thus, while our results are motivated by \cite{AlemanPerfektRichterSundbergSunkes2023} and \cite{Richteriteratedlog},  our theorems will be valid for weighted Besov spaces $ B^N_\omega$ that lie between the Drury--Arveson space $H^2_d$ and the Hardy space $H^2(\dB)$ a range of spaces that are known to not have the Pick property.

Central to our work is the concept of \emph{weak$^\ast$sequential cyclicity}.

\begin{defn}
A function $\phi \in \mathrm{Mult}(H)$ is said to be 
\textbf{\emph{weak$^\ast$ sequentially cyclic}} if there exists a sequence 
$\{u_n\} \subseteq  \mathrm{Mult}(H)$ such that
\[
u_n\phi \xrightarrow{\;\mathrm{weak}^\ast\;} 1.
\]
\end{defn}

The spaces $H=B^N_\omega$ that we will consider, all contain their multipliers as a dense subset. In that case this notion is formally stronger than  cyclicity: if $\phi \in \mathrm{Mult}(B^{N}_{\omega})$ is 
weak$^\ast$ sequentially cyclic, then $\phi$ is cyclic in 
$B^{N}_{\omega}$ (see Lemma \ref{lem:*cyclic is cyclic}). While the converse holds for complete Pick spaces (see Corollary \ref{cor:Pick}), 
in Section \ref{Sec:Example} we will provide an 
 example  showing that for $d\ge 2$ there are $H^\infty(\dB)$-functions that are cyclic in $H^2(\dB)$, but 
not  weak$^\ast$ sequentially cyclic in $H^\infty(\dB)$. Thus, the new concept naturally fits within the existing 
framework, and it has the potential for additional structural control.

We will see that weak$^\ast$ sequential cyclicity has many of the same elementary properties that cyclicity has: Under suitable 
hypotheses weak$^\ast$ sequential cyclicity is preserved under 
products and quotients by weak$^\ast$ sequentially cyclic multipliers.  

We will now restrict our discussion to the weighted Besov spaces  $H=B^N_\omega$. In that context  we will show the following result, see Theorem \ref{thm:f less than g}. 

\begin{thm} If $\phi, \psi \in \mathrm{Mult}(B^N_\omega)$ with $|\phi(z)|\le |\psi(z)|$ for all $z\in \Bd$, and if $\phi$ is weak$^\ast$ sequentially cyclic, then so is $\psi$.\end{thm} 
If $ B^N_\omega$ happens to be a Pick space, then this result also follows from the equivalence of the cyclicity concepts mentioned above and Theorem 3.1 of \cite{AlemanPerfektRichterSundbergSunkes2023}. In fact, our proof  follows the strategy of \cite{AlemanPerfektRichterSundbergSunkes2023}.

In  \cite{AlemanHartzMcCarthyRichter2017} the authors defined the Pick-Smirnov class of a Pick space $H$ by $$N^{+}(H)
=
\left\{
\frac{u}{v} \;:\; u, v \in \mathrm{Mult}(B^{N}_{\omega}),
\ \text{and } v \text{ is  cyclic in }H
\right\}.$$

Analogously, we define
\begin{equation}
N^{+}_{\ast}(B^{N}_{\omega})
=
\left\{
\frac{u}{v} \;:\; u, v \in \mathrm{Mult}(B^{N}_{\omega}),
\ \text{and } v \text{ is weak$^\ast$ sequentially cyclic}
\right\}
\end{equation}
Thus, $N^{+}_{\ast}(B^{N}_{\omega})\subseteq N^{+}(B^{N}_{\omega})$ with equality, whenever $B^N_\omega$ is a Pick space. Since cyclic functions have no zeros, it is clear that functions in $N^{+}_{\ast}(B^{N}_{\omega})$ are analytic in $\Bd$. Furthermore, it is elementary to show that a multiplier $\phi$ is weak$^\ast$ sequentially cyclic, if and only if $1/\phi\in N^{+}_{\ast}(B^{N}_{\omega})$. This may be considered to be a tautology, but it shifts the viewpoint from smallness and boundary zero sets of multipliers to boundary growth and exceptional sets of functions in $N^{+}_{\ast}(B^{N}_{\omega})$.

We will see that if  $\phi\in \mathrm{Mult}(B^{N}_{\omega})$ has  no zeros in  $\mathbb{B}_d$ and if it  has bounded argument, then  $\log \phi $ is a ratio of multipliers with no zeros in $\mathbb{B}_d$, see Theorem \ref{thm:log=ratio}. One might say that $\log \phi$ is in the Pick-Nevanlinna class. Additionally, if such $\phi$ is $weak^\ast$-sequentially cyclic, then $\log \phi \in N^{+}_{\ast}(B^{N}_{\omega})$. 
Our main result is a weak converse of this.

\begin{thm}
Let $f \in B^{N}_{\omega} \cap H^\infty$ and assume $f$ has no zeros in 
$\mathbb{B}_d$ and has bounded argument. 

  If ~ $ \log f \in N^{+}_{\ast}(B^{N}_{\omega}),$
then $f$ is cyclic in $B^N_\omega$. In addition,
$\frac{1}{f} \in N^{+}(B^{N}_{\omega})$ and $f \in N^{+}_{\ast}(B^{N}_{\omega})$.
\end{thm}

This will be Theorem \ref{thm:logN+*}. If $B^N_\omega$ is a Pick space, then $N^{+}_{\ast}(B^{N}_{\omega})=N^{+}(B^{N}_{\omega})$, and this Theorem is covered by Theorem 1.1 of \cite{Richteriteratedlog}. In fact, in that case such $f$ is cyclic in $B^N_\omega$, if and only if $\log f\in N^+(B^N_\omega)$.  
In particular, in \cite{AlemanRichterDirichletspace2026}, it has been shown that this necessary and sufficient condition holds for functions in the Dirichlet space on the disc even without the hypothesis of bounded argument.

In one of our proofs in Section 3 we have used Lemma 3.3 of \cite{AlemanHartzMcCarthyRichter2023}. In the Appendix we included a shortened proof of that theorem.

\section{The weighted Besov spaces \texorpdfstring{$B^N_\omega$}{BN omega}}
Let $d $ be a positive integer, and consider the open unit ball $\mathbb{B}_d \subseteq  \mathbb{C}^d$. Let $Hol(\mathbb{B}_d)$ denote  the space of holomorphic functions on $\mathbb{B}_d$. We define a \textbf{\emph{radially weighted Besov space}} associated with a radial measure $\omega$ and a non-negative integer $N$ by
\[
B^N_\omega := \bigl\{ f \in Hol(\mathbb{B}_d) : R^N f \in L^2(\omega) \bigr\},
\]
where $R$ is the radial derivative operator
\[
R = \sum_{j=1}^{d} z_j \frac{\partial}{\partial z_j}
\]
and  $\omega$ is assumed to be an admissible  radial measure on $\overline{\Bd}$. This means  it can be expressed as
\[
d\omega(z) = d\mu(r) \, d\sigma(w), \quad z = r w,
\]
with $\sigma$ the normalized rotation-invariant measure on the sphere $\partial \mathbb{B}_d$, and $\mu$ a Borel measure on $[0,1]$ such that
\[
\mu((r,1]) > 0 \quad \text{for all } 0<r<1.
\]

\medskip

\noindent
\textbf{Remark.} If $\mu$ has a point mass at $r=1$, then we define the $L^2(\omega)$-norm of a function $f \in Hol(\mathbb{B}_d)$ by
\[
\|f\|^2_{L^2(\omega)} = \int_{\mathbb{B}_d} |f|^2 \, d\omega + \mu(\{1\}) \, \|f\|^2_{H^2(\partial \mathbb{B}_d)},
\]
where $H^2(\partial \mathbb{B}_d)$ denotes the Hardy space of holomorphic functions $f$ satisfying
\[
\|f\|^2_{H^2(\partial \mathbb{B}_d)} = \sup_{0 \le r < 1} \int_{\partial \mathbb{B}_d} |f(r z)|^2 \, d\sigma(z) = \lim_{r \to 1^-} \int_{\partial \mathbb{B}_d} |f(r z)|^2 \, d\sigma(z) < \infty.
\]

\medskip

\noindent
The space $B^N_\omega$ is equipped with the norm
\[
\|f\|^2_{B^N_\omega} =
\begin{cases}
\|f\|^2_{L^2(\omega)}, & N=0,\\[0.5em]
\omega(\mathbb{B}_d) \, |f(0)|^2 + \|R^N f\|^2_{L^2(\omega)}, & N>0.
\end{cases}
\]

\medskip

\noindent
Under these definitions, $B^N_\omega$ becomes a Hilbert function space over $\mathbb{B}_d$, in the sense that point evaluations are continuous linear functionals.

\medskip

For example if $d\omega = d\sigma \, d\delta_{1}$. Then the scale of spaces
$\{B^{N}_{\omega}\}_{N \ge 0}$ satisfies the following inclusions,
according to the value of $N$ relative to $\frac{d-1}{2}$, see \cite{Richteriteratedlog}.

\medskip
\noindent\textbf{The Pick space case ($N \geq \frac{d-1}{2}$):}
In this range, $B^{N}_{\omega}$ is a  Pick space and we have
\[
B^{N}_{\omega}
\subseteq 
H^2_d.
\]
\medskip
\noindent\textbf{The Non-Pick space case ($0 \le N < \frac{d-1}{2}$):}
In this case, $B^{N}_{\omega}$ is not a Pick space and we obtain
\[
H^2_d
\subseteq 
B^{N}_{\omega}
\subseteq 
B^{0}_{\omega}
=
H^2(\partial \mathbb{B}_d).
\]
 In the present work, our main interest is on those spaces which are not Pick spaces, although all of the results we obtain also hold in the setting of  Pick spaces. In fact, as mentioned in the Introduction, for  Pick spaces more general results have been established in \cite{ AlemanPerfektRichterSundbergSunkes2023, Richteriteratedlog}. 
 
 We will use the following two theorems from \cite{Richteriteratedlog}: 

\begin{thm}[Theorem 4.2 in \cite{Richteriteratedlog}]\label{thm:Iterated_thm4.2}
Let $N \in \mathbb{N}$ and $\omega$ be an admissible radial measure. 
 There exist constants $C > 0$ and $n \in \mathbb{N}$ such that whenever $\phi, \psi \in \mathrm{Mult}(B^N_{\omega})$ with
\[
\operatorname{Re}\!\left(\frac{\phi}{\psi}\right) \geq 0,
\]
then
\[
\psi^n e^{-\frac{\phi}{\psi}} \in \mathrm{Mult}(B^N_{\omega}),
\]
and
\[
\left\| \psi^n e^{-\frac{\phi}{\psi}} \right\|_{\mathrm{Mult}(B^N_{\omega})}
\leq
C \left( \|\psi\|_{\mathrm{Mult}(B^N_{\omega})} + \|\phi\|_{\mathrm{Mult}(B^N_{\omega})} \right)^n.
\]
\end{thm}

\begin{thm} [Lemma 4.1(b) in \cite{Richteriteratedlog}]\label{thm:Iterated_Lem4.1}
Let $N \in \mathbb{N}$, $\omega$ be an admissible radial measure.

\medskip

 If $\phi \in \mathrm{Mult}(B^N_{\omega})$ is zero-free in $\mathbb{B}_d$ and $j$ is a non-negative integer such that
\[
\phi^j \log \phi \in H^{\infty},
\]
then there exists an integer $n \geq j$ such that
\[
\phi^n \log \phi \in \mathrm{Mult}(B^N_{\omega}).
\]
\end{thm}
 
\section{Multipliers and \texorpdfstring{$weak^\ast$}{weak*} closed ideals}
 Let $\HH$ be a reproducing kernel Hilbert space of analytic functions on $\Bd$. It is clear that linear combinations of certain countable sets of reproducing kernels are dense in $H$, hence $H$ is separable. Since $H$ is separable,  the trace class operators on $H$ are a separable pre-dual of $\HB(H)$, the algebra of all bounded linear transformations on $H$. We write $\MH$ for the multiplier algebra of $H$. Then each element $\phi$ of $\MH$ induces a bounded multiplication operator $M_\phi\in \HB(H)$ and $\MH$ becomes a Banach space with norm $\|\phi\|
_{Mult(H)}=\|M_\phi\|_{\HB(H)}= 
\sup
\bigl\{
\|\phi f\|
:
f \in H,\,
\|f\|\le 1
\bigr\}.$  
$\MH$ is thus  isometrically isomorphic to a $weak^\ast$-closed sub-algebra of $\mathcal{B}(H)$, and it can be identified with the dual of a separable Banach space. As such, it carries a $weak^\ast$-topology, which agrees with the $weak^\ast$-topology of $\mathcal{B}(\HH)$ restricted to $\MH$. One easily checks that point evaluations on $\Bd$ are $weak^\ast$-continuous linear functionals on $\MH$. Then it is a standard fact that a sequence of multipliers $\phi_n \in \MH$ converges in the $weak^\ast$-topology to $\phi\in \MH$, if and only if $\phi_n(z) \to \phi(z)$ for all $z\in \Bd$ and there is $M>0$ such that $\|\phi_n\|_{\MH}\le M$ for all $n \in \mathbb{N}$, (\cite{BrownShields1984}, {Proposition 2}).
 
 If $\phi\in \MH$, then $\MH \phi= \{u\phi: u\in \MH\}$ and we write $$[\phi]_*=\overline{\MH \phi}^{wk^*}\subseteq \MH$$ for the $weak^\ast$-closed ideal generated by $\phi$. We have been unable to answer the following question in all generality.
 
 \begin{quest}  \label{ques:weak*}
 If $\phi, \psi \in \MH$ such that $\psi \in [\phi]_*$, then is there a sequence $u_n\in \MH$ such that $u_n \phi \to \psi$ in the $weak^\ast$-topology?\end{quest}

  If $\HH$ has a normalized Pick kernel, then the answer is yes. That can be deduced  from a result of Davidson, Ramsey, and Shalit, \cite{DavidsonRamseyShalit2015}, Corollary 2.7, which says that for Pick spaces $H$ the $weak^\ast$-closed ideals in $\MH$ are in 1-1 correspondence with the multiplier invariant subspaces of $\HH$. See Corollary \ref{cor:Pick} below for a more direct proof that is based on Lemma 3.3 of \cite{AlemanHartzMcCarthyRichter2023}. However, we do not know the answer for example for $H^\infty(\Bd)$, if $d \ge 2$. Thus, it remains unclear under which conditions a multiplier $\phi$ that satisfies $[\phi]_*=\MH$ is $weak^\ast$ sequentially cyclic.
  
  Let $\mathcal S(\phi)$ be the $weak^\ast$-sequential closure of $\MH \phi$. Then it follows from the Krein-Smulian Theorem (see \cite{Conway1990}), Corollary V.12.7) that $\mathcal S(\phi)$ is $weak^\ast$-closed, if and only if it is $weak^\ast$-sequentially closed. Thus,  the above Question is equivalent to the question of whether $\mathcal S(\phi)$ is $weak^\ast$-sequentially closed. The  structure of $\MH$  can be used to restate the question.
  
  \begin{lem}  \label{lem:weak*} 
  Let $\phi\in \MH$. Then $$\mathcal S(\phi)=[\phi]_*$$ if and only if there is $c>0$ such that for all $\psi\in \mathcal S(\phi)$ with $\|\psi\|_{\MH}\le 1$ there is a sequence $u_n\in \MH$ such that $ \|u_n\phi\|_{\MH}\le c$ for all $n\in \mathbb
  {N}$ and $u_n(z) \phi(z) \to \psi(z)$ for all $z\in \Bd$.
  \end{lem}
  \begin{proof} The Lemma is a strengthening of one of the directions of the equivalence of (a) and (c) of Theorem V.12.11 of \cite{Conway1990}. Indeed, if $\mathcal S(\phi)=[\phi]_*$, then $\MH\phi$ is sequentially dense in $[\phi]_*$,  hence $(a)\Rightarrow (c)$ of  Theorem V.12.11 of \cite{Conway1990} implies the conclusion as stated in  the Lemma.

  Conversely, assume that there is $c>0$ such that for all $\psi$ in the unit ball of $\mathcal S(\phi)$ there is a sequence $u_n\in \MH$ such that $u_n(z) \phi(z) \to \psi(z)$ for all $z\in \Bd$ and $\|u_n\phi\|_{\MH}\le c$ for all $n$.
  
  By what was stated before the Lemma, we have to show that $\mathcal S(\phi)$ is $weak^\ast$-sequentially closed. Thus, let $ \psi_n \in \mathcal S(\phi)$ and $\psi\in \MH$ such that $\psi_n \to \psi$ in the $weak^\ast$-topology. Thus, there is $M>0$ such that $\|\psi_n\|_{\MH}\le M$ for all $n$ and $\psi_n \to \psi$ pointwise in $\Bd$. Choose a sequence of positive reals $1>r_n \to 1$.

  Fix $n\in \mathbb{N}$. By our hypothesis and a scaling argument there is a sequence $\{u_{nk}\}_{k\in \mathbb{N}}\in \MH$ such that $\|u_{nk}\phi\|_{\MH}\le cM$ and $u_{nk}\phi\to \psi_n$ pointwise in $\Bd$ as $k\to \infty$.
  The unit ball in $\MH$ is a normal family, hence the pointwise convergence in $\Bd$ for norm bounded sequences is equivalent to local uniform convergence. Thus,  there is $k_n\in \mathbb{N}$ such that $|u_{nk_n}(z)\phi(z)-\psi_n(z)|<1/n$ for all $z\in r_n\Bd$. 
  
  Then $u_{nk_n}\phi \to \psi$ in the $weak^\ast$-topology. Indeed, we already noted the uniform boundedness of the norms of $u_{nk_n}\phi$, and if $z\in \Bd$, then for any $n$ such that $|z|<r_n$ we have 
  $$|u_{nk_n}\phi(z)-\psi(z)|\le \frac{1}{n}+|\psi_n(z)-\psi(z)|.$$ Thus, $u_{nk_n}\phi \to \psi$ pointwise in $\Bd$.
  \end{proof}
\begin{cor}
    
\label{cor:Pick} If $H\subseteq \mathrm{Hol}(\Bd)$ has a normalized Pick kernel, then $\mathcal S(\phi)=[\phi]_*$ for all $\phi\in \MH$. Thus, $\phi$ is $weak^\ast$ sequentially cyclic, if and only if $1\in [\phi]_\ast$.
\end{cor}
\begin{proof} Let $\phi\in \MH$, then $\phi\in H$ since $H$ has a normalized Pick kernel. Hence we define $[\phi]$ to be the closure in $H$ of $\MH \phi$. Since $weak^\ast$-convergent sequences in $\MH$ are weakly convergent in $H$, and $[\phi]$ is weakly closed, it follows that $\mathcal S(\phi)\subseteq \MH\cap [\phi]$. Thus, Lemma 3.3 of \cite{AlemanHartzMcCarthyRichter2023} implies that $\phi$ satisfies the condition of Lemma \ref{lem:weak*} with $c=1$.
\end{proof}
Notice that Lemma 3.3 of \cite{AlemanHartzMcCarthyRichter2023} also implies that if $H$ is a Pick space, then multipliers that are cyclic in $H$ must be $weak^\ast$ sequentially cyclic.
    
\section{Some preliminary results}
\begin{lem}\label{lem:*cyclic is cyclic} Let  $\phi \in \mathrm{Mult}(B^{N}_{\omega})$. If $1\in [\phi]_\ast$, then $\phi$ is cyclic in $B^{N}_{\omega}$.

Consequently  weak$^\ast$ sequentially cyclic multipliers are cyclic  in $B^{N}_{\omega}$.
\end{lem}

\begin{proof}
Suppose  $1\in [\phi]_\ast$, then there is a net of multipliers $\{\psi_\gamma\}_\gamma$ such that $\psi_\gamma \phi \to 1$ in the $weak^\ast$ topology of $\mathrm{Mult}(B^{N}_{\omega})$. If $g\in B^N_\omega$, then $L(\psi)=\la \psi , g\ra_{B^N_\omega}= \la M_\psi 1, g\ra_{B^N_\omega}$ defines a $weak^\ast$-continuous linear functional on $\mathrm{Mult}(B^N_\omega)$. Hence 
$$\la \psi_\gamma \phi,g\ra_{B^N_\omega}=L(\psi_\gamma \phi)\to L(1)=\la 1,g\ra_{B^N_\omega},$$
i.e. $\psi_\gamma \phi \to 1$ weakly in $B^N_\omega$. We write $[\phi]$ for the smallest invariant subspace of $B^N_\omega$ that contains $\phi$. By standard Banach space theory, it agrees with the weak closure of the multiplier multiples of $\phi$, hence  $1\in [\phi]$, and we conclude that $\phi$ is cyclic in $B^N_\omega$.
\end{proof}

We know the following lemma is true, see \cite{BrownShields1984}, Proposition 11.
\begin{lem}\label{lem:cyclicProducts}
    If $f \in B^{N}_{\omega}$ and $g \in Mult(B^{N}_{\omega})$, then $gf$ is cyclic if and only if both $f$ and $g$ are cyclic.
\end{lem}

A similar  result holds for  weak$^\ast$ sequential cyclicity.

\begin{lem}
Let $f, g \in Mult(B^{N}_{\omega})$. Then $fg$ is weak$^\ast$ sequentially cyclic if and only if both $f$ and $g$ are weak$^\ast$ sequentially cyclic.
\end{lem}

\begin{proof}
First, assume that $fg$ is weak$^\ast$ sequentially cyclic. Then there exists a sequence $\{p_n\} \subseteq  Mult(B^{N}_{\omega})$ such that
\[
p_n (fg) \to 1 \quad \text{$weak^\ast$}.
\]
This implies $(p_n f) g \to 1$ $weak^\ast$. By hypothesis, $p_n f \in Mult(B^{N}_{\omega})$, so $g$ is weak$^\ast$ sequentially cyclic.  

Also, by symmetry, $f $ is weak$^\ast$ sequentially cyclic as well.  

Conversely, assume that $f$ and $g$ are both weak$^\ast$ sequentially cyclic in $B^{N}_{\omega}$. Then there exist sequences $(p_k), (q_k) \subseteq  Mult(B^{N}_{\omega})$ such that
\[
p_k f \to 1 \quad \text{$weak^\ast$}, \quad q_k g \to 1 \quad \text{$weak^\ast$}.
\]
There exist constants $C, C' > 0$ such that
\[
\|p_k f\|_{Mult(B^{N}_{\omega})} \leq C, \quad \|q_k g\|_{Mult(B^{N}_{\omega})} \leq C'.
\]
Hence,
\[
(q_k p_k) (fg) \to 1 \quad \text{pointwise}, \quad \|(q_k p_k) (fg)\|_{Mult(B^{N}_{\omega})} \leq M,
\]
for some $M = CC' > 0$. Therefore, $fg$ is weak$^\ast$ sequentially cyclic.
\end{proof}

Note that Lemma 4.1 implies that $N^{+}_{*} (B^{N}_{\omega})$ is an algebra. We also obtain the following corollaries.

\begin{cor}\label{cor:u^n}
If $u \in Mult(B^{N}_{\omega})$ is weak$^\ast$ sequentially cyclic, then for any $n \in \mathbb{N}$, $u^n$ is weak$^\ast$ sequentially cyclic.
\end{cor}
\begin{cor}
If $\phi \in Mult(B^{N}_{\omega})$, then  $\phi$ is weak$^\ast$ sequentially cyclic, if and only if $1/\phi\in N^+_\ast(B^N_\omega)$.
\end{cor}
\begin{proof} Suppose $\frac{1}{\phi}=\frac{u}{v}$ for some multipliers $u$ and $v$, where $v$ is weak$^\ast$ sequentially cyclicity.. Then $v=u\phi$, hence $\phi$ must be weak$^\ast$ sequentially cyclic. The other direction is trivial.\end{proof}
\section{The comparison principle}
 \begin{thm}\label{thm:f less than g}
If $f, g \in \mathrm{Mult}(B^{N}_{\omega})$ satisfy $|f| \leq |g|$, and if $f$ is weak$^\ast$ sequentially cyclic, then $g$ is weak$^\ast$ sequentially cyclic.
\end{thm}

\begin{proof}
Since $f, g \in \mathrm{Mult}(B^{N}_{\omega})$ with $|f| \leq |g|$, by Theorem 3.2 in \cite{AlemanPerfektRichterSundbergSunkes2023} there exists $n \in \mathbb{N}$ such that
\[
\frac{f^{n}}{g} \in \mathrm{Mult}(B^{N}_{\omega}).
\]

Because $f$ is weak$^\ast$ sequentially cyclic, there exists a sequence $\{\phi_{k}\} \subseteq  \mathrm{Mult}(B^{N}_{\omega})$ such that
\[
\phi_{k} f \to 1
\quad \text{weak}^\ast.
\]
Hence,
\[
(\phi_{k} f)^{n} \to 1
\quad \text{weak}^\ast.
\]
This implies that
\[
\left( \phi_{k}^{\,n} \frac{f^{n}}{g} \right) g
=
\phi_{k}^{\,n} f^{n}
\to 1
\quad \text{weak}^\ast.
\]
Observe that
\[
\phi_{k}^{\,n} \frac{f^{n}}{g}
\in \mathrm{Mult}(B^{N}_{\omega}).
\]
Therefore, $g$ is weak$^\ast$ sequentially cyclic.
\end{proof}

\section{Theorems for \texorpdfstring{$\log f$}{log f}}
We are interested in cyclic functions, thus we will always want to assume that $\phi(z)\ne 0$ in $\Bd$. Then $\phi$ has an analytic logarithm in $\Bd$ and saying that  $\phi$ has bounded argument is equivalent to saying that $g=\log \phi$ has bounded imaginary part. In \cite{Richteriteratedlog} it was shown that polynomials without zeros in the ball always have bounded argument. In single variable we know that functions with bounded real part satisfy an inequality of the type $|g(z)|\le c\log\frac{1+|z|}{1-|z|}$ (See \cite{Rudinbookrealandcomplex}, Page 294 - ex:9 and 10). By considering slice functions, the same conclusion holds for functions in $\Bd$ and we conclude that if $\phi$ has bounded argument, then there is $c>0$ such that  $|\phi(z)|\ge A {(1-|z|)^c}$ for all $z\in \Bd$.

We start with the following observation.
\begin{thm}\label{thm:log=ratio}
    Let $\phi \in \mathrm{Mult}(B^{N}_{\omega})$ have no zeros in  $\mathbb{B}_d$ and bounded argument.
    Then there are $ \psi \in \Mult(B^N_\omega)$ and $n\in \N$ such that $\log \phi= \frac{\psi}{\phi^n}$. 
    
    Consequently, if $\phi$ is $weak^\ast$-sequentially cyclic, then $\log \phi \in N^+_\ast(B^N_\omega)$.
\end{thm}
\begin{proof}
    Since multipliers are bounded, there  exists a constant $C > 0$ such that
\[
|\phi(z) \log |\phi(z)|| \leq C \quad \text{for all } z \in B^{N}_{\omega}.
\]
By hypothesis, $\phi$ has bounded argument, so there exists $A > 0$ with
\[
|\Im (\log \phi(z))| \leq A \quad \text{for all } z \in B^{N}_{\omega}.
\]
Thus,
\[
|\phi(z) \log \phi(z)| \leq |\phi(z)| \log|\phi(z)| + |\phi(z)| |\arg \phi(z)| \leq C + A \|\phi\|_\infty < C',
\]
for some constant $C' > 0$. This implies $\phi \log \phi \in H^\infty$.  

Hence by Theorem \ref{thm:Iterated_Lem4.1}, there exists $n \in \mathbb{N}$ such that 
\[
\psi:= \phi^n \log \phi \in Mult(B^{N}_{\omega}).
\]
Thus, $\log \phi= \frac{\psi}{\phi^n}$.
Now assume that $\phi$ is $weak^\ast$-sequentially cyclic, then by Corollary \ref{cor:u^n} the function $\phi^n$ is $weak^\ast$-sequentially cyclic, hence  $\log \phi\in N^+_\ast(B^N_\omega)$.
\end{proof}

\begin{lem}\label{lem:weak*log}
Let $f \in B^{N}_{\omega} \cap H^{\infty}$ and assume that there are  $u,v\in \mathrm{Mult}(B^N_\omega)$ such that
\begin{enumerate}
\item $\|f\|_\infty\le 1$,
\item $f=e^{-\frac{u}{v}}$,
\item $f$ has bounded argument, and
\item $v$ is $weak^\ast$ sequentially cyclic.
\end{enumerate}
 Then there is $n\in\N$ such that $v^n f\in \mathrm{Mult}(B^N_\omega)$ and
\[
 [v^n f]_{*} = Mult(B^{N}_{\omega}).
\]
\end{lem}

\begin{proof} Fix $0\le \alpha \le 1$. Since $\|f\|_{\infty} \leq 1$, we have 
\[
\Re\!\left(\alpha \frac{u}{v}\right) \geq 0.
\]

By Theorem 2.1, there exists $n \in \mathbb{N}$ and $C>0$ such that 
$ v^{n} f^\alpha \in \mathrm{Mult}(B^{N}_{\omega})$ with 
\begin{align}\label{equ:v^nf}
 \|v^n f^\alpha\|_{\mathrm{Mult}(B^{N}_{\omega})}&\le C(\|\alpha u\|_{\mathrm{Mult}(B^{N}_{\omega})}+\|v\|_{\mathrm{Mult}(B^{N}_{\omega})})^n \nonumber \\
 &\le C(\| u\|_{\mathrm{Mult}(B^{N}_{\omega})}+\|v\|_{\mathrm{Mult}(B^{N}_{\omega})})^n
\end{align}

Since $v$ is weak$^\ast$ sequentially cyclic, there exists a sequence $\{\phi_{j}\} \subseteq  \mathrm{Mult}(B^{N}_{\omega})$ such that
\[
\phi_{j} v \to 1 \quad \text{pointwise}
\]
and there is $M>0$ such that
\[
\|\phi_{j} v\|_{\mathrm{Mult}(B^{N}_{\omega})} \leq M
\]
 for all $ j \in \mathbb{N}$. Moreover, because $f$ has bounded argument, there exists $A > 0$ such that
\[
|\Im \frac{u}{v}|=|\Im \log f| \leq A.
\]

Note that
\[
\|\phi_{j} v\|_{\infty}
\leq
\|\phi_{j} v\|_{\mathrm{Mult}(B^{N}_{\omega})}
\leq M.
\]
Since $\phi_jv(z)\to 1$ we must have $M\ge 1$. 
Set
\[
\gamma = \frac{1}{M}>0.
\]

\noindent{\underline{Claim:}}  $v^n f^{\alpha(1-\gamma)}\in [v^n f^\alpha]_\ast$.

\

Note that $\psi_j= e^{\alpha \gamma \phi_{j} u} \in \mathrm{Mult}(B^{N}_{\omega})$. We will establish the claim by showing that $\psi_j v^n f^\alpha \to v^n f^{\alpha(1-\gamma)}$ in the $weak^\ast$ topology of $\mathrm{Mult}(B^{N}_{\omega})$. 

Since $$\psi_jf^\alpha= e^{-\alpha \frac{u}{v}(1-\gamma \phi_j v)}$$ the pointwise convergence follows from the pointwise convergence of $\phi_jv\to 1$. We have to prove that $\|\psi_jv^nf^\alpha\|_{\mathrm{Mult}(B^{N}_{\omega})}$ is bounded independently of $j$.

For all $j\in \N$ we have
$$\Re(1 - \gamma \phi_{j} v(z)) \geq 0 \ \text{ and } \ \
|\Im(1 - \gamma \phi_{j} v(z))| \leq 1.$$

This implies
\begin{align*}
\Re\!\left( \frac{u}{v} (1 - \gamma \phi_{j} v) \right)
&=
\Re\!\left(  \frac{u}{v} \right)
\Re(1 - \gamma \phi_{j} v)
-
\Im\!\left( \frac{u}{v} \right)
\Im(1 - \gamma \phi_{j} v) \\
&\geq - \,  A.
\end{align*}
Hence 
\[
\Re\!\left(
\alpha \frac{u}{v} (1 - \gamma \phi_{j} v)
+
\alpha \frac{ A v}{v}
\right)
\geq 0.
\]

Now apply Theorem 2.1 again, this time with
\[
\phi = \alpha u (1 - \gamma \phi_{j} v) + \alpha  A v
\text{ and }
\psi = v.
\]
Note that the $n$ and $C$ in Theorem \ref{thm:Iterated_thm4.2} do not depend on $\phi$, $\psi$, or $\alpha$. Thus

\begin{align*} \|\psi_jf^\alpha\|_{\mathrm{Mult}(B^{N}_{\omega})}
&=e^{\alpha A} \left\| v^{n}
e^{-\frac{\alpha u (1 - \gamma \phi_{j} v) + \alpha  A v}{v}}
\right\|_{\mathrm{Mult}(B^{N}_{\omega})}\\  
&\leq C e^{\alpha A} 
\left(\|v\|_{\mathrm{Mult}(B^{N}_{\omega})}
+ \|\alpha u (1 - \gamma \phi_{j} v) + \alpha  A v\|_{\mathrm{Mult}(B^{N}_{\omega})}
\right)^{n} \\   
&\leq C e^{\alpha A} \left(
\|v\|_{\mathrm{Mult}(B^{N}_{\omega})}(1 +  A) +
(1 + \gamma M)\|u\|_{\mathrm{Mult}(B^{N}_{\omega})}
\right)^{n} <\infty.
\end{align*}
This establishes the Claim.

For $k=0,1, \dots$  let $\alpha_{k} = (1 - \gamma)^{k}.$ Then the Claim implies that for each $k$ we have $v^n f^{\alpha_{k+1}}\in [v^n f^{\alpha_k}]_\ast$.
Since for each $\phi\in \mathrm{Mult}(B^N_\omega)$ the subspace $[\phi]_\ast$ is multiplier invariant, we obtain inductively that $[v^n f^{\alpha_k}]_\ast\subseteq [v^nf]_\ast$ for all $k$. But $\alpha_k\to 0$ and $\|v^n f^{\alpha_k}\|_{\mathrm{Mult}(B^{N}_{\omega})}$ is bounded independently of $k$ by inequality (\ref{equ:v^nf}). Thus, $v^n f^{\alpha_k} \to v^n$ $weak^\ast$, hence $[v^n]_\ast \subseteq [v^nf]_\ast$. But now Corollary \ref{cor:u^n} implies that $[v^n]_\ast = \mathrm{Mult}(B^N_\omega)$ and this concludes the proof of the Lemma.
\end{proof}

    \begin{thm}\label{thm:logN+*} Let $f \in B^{N}_{\omega} \cap H^{\infty}$ and assume that $f$ has no zeros  inside $\mathbb{B}_d$ and has bounded argument. If
\[
\log f \in N^{+}_{*}(B^{N}_{\omega}),
\]
then $f$ is cyclic in $ B^{N}_{\omega}$.\\
Moreover,
\[
f \in N^{+}_{*}(B^{N}_{\omega}) \quad \text{and} \quad 
\frac{1}{f} \in N^{+}(B^{N}_{\omega}).
\]
 \end{thm}

\begin{proof} Let $f$ be as in the hypothesis of the theorem. We may assume that $\|f\|_\infty \le 1$. Then $f$ satisfies the hypothesis of Lemma \ref{lem:weak*log} and we conclude that there is a $n\in \N$ and a $weak^\ast$ sequentially cyclic multiplier $v$ such that $\phi=v^nf\in \mathrm{Mult}(B^N_\omega)$ and $1\in [v^n f]_\ast$. Then Lemma \ref{lem:weak*} implies that $\phi=v^nf$ is cyclic in $B^N_\omega$, and by Lemma \ref{lem:cyclicProducts} we can conclude that $f$ is cyclic in $B^N_\omega$. 

Furthermore, $f=\phi/v^n\in N^+_\ast(B^N_\omega)$ and $\frac{1}{f}=\frac{v^n}{v^n f}\in N^+(B^N_\omega)$.
\end{proof}

\section{Example: Cyclicity Does Not Imply \texorpdfstring{weak* sequential cyclicity}{weak$^\ast$ sequential cyclicity}}
\label{Sec:Example}

Let $S(z)$ be a  non-trivial singular inner function that is cyclic in the Bergman space of the unit disc $\D$, $$L^{2}_{a}=\{f\in \mathrm{Hol}(\D): \int_{\D}|f|^2 \frac{dxdy}{\pi}<\infty\}.$$ Such a function exists, see  \cite{KorenblumBeurlingtype, Korenblumcyclicelements, Robertscyclicinner}). Note that $\mathrm{Mult}(H^2(dB))=H^\infty(\Bd)$.

\begin{thm} The function $f(z,w)=S(z)\in H^\infty(\mathbb B_2)$ is cyclic in $H^2(\partial \mathbb B_2)$, but it is not $weak^\ast$ sequentially cyclic in $H^\infty(\mathbb B_2)$. 
\end{thm}

\begin{proof} Using multinomial notation the reproducing kernel for $H^2(\partial \mathbb B_2)$ is $$k_\lambda(z)=\frac{1}{(1-\la z,\lambda\ra)^2}=\sum_{n=0}^\infty (n+1) \sum_{|\alpha|=n} \frac{n!}{\alpha !} \overline{\lambda}^\alpha z^\alpha.$$
Hence if $h\in \mathrm{Hol}(\partial \mathbb B_2)$ has Taylor series 
$$h(z,w)=\sum_{n=0}^\infty \sum_{\alpha_1+\alpha_2=n} \hat h(\alpha_1,\alpha_2) z^{\alpha_1}w^{\alpha_2},$$
then for $g(z)=h(z,0)$ we have
\begin{align*}\|h\|^2_{H^2(\partial \mathbb B_2)}&= \sum_{n=0}^\infty \frac{1}{n+1}\sum_{\alpha_1+\alpha_2=n} \frac{\alpha_1! \alpha_2!}{n!}| \hat h(\alpha_1,\alpha_2)|^2\\
&\ge   \sum_{n=0}^\infty \frac{1}{n+1}  |\hat h(n,0)|^2 = \|g\|^2_{L^2_a},\end{align*}
and equality holds, whenever $h(z,0)=h(z,w)$ for all $(z,w)\in \mathbb B_2$.

Since $S$ is cyclic in $L^2_a$, there is a sequence of polynomials $\{p_{n}(z)\}$ such that
\[
\| p_{n} S- 1 \|_{L^{2}_{a}} \longrightarrow 0.
\]
Then with $q_n(z,w)=p_n(z)$ we have $\|q_nf-1\|^2_{H^2(\partial \mathbb B_2)}=\| p_{n} S- 1 \|_{L^{2}_{a}} \longrightarrow 0$, hence $f$ is cyclic in $H^2(\partial \mathbb{B}_{2})$.

Now suppose, for contradiction, that $f$ is also $weak^\ast$ sequentially cyclic in $H^\infty(\mathbb B_2)$. Then there exists a sequence of polynomials $\{q_{n}(z,w)\}$ such that
\[
q_{n} f \longrightarrow 1
\quad \text{in the weak}^\ast \text{ topology of } H^\infty(\{ |z|^{2} + |w|^{2} < 1 \}).
\]

Write $p_n(z)=q_n(z,0)$ and observe that
\[
\sup_{|z|^{2}+|w|^{2}<1} |q_{n}(z,w) f(z,w)|
\geq
\sup_{|z|<1} |q_{n}(z,0) f(z,0)|
=
\sup_{|z|<1} |p_{n}(z) S(z)|.
\]
Hence, $p_{n} S \to 1$ in the $weak^\ast$ sense on $\D=\{|z|<1\}$. 
This implies that $S$ is $weak^\ast$-sequentially cyclic in $H^{2}(\D)$.

However, in $H^{2}(\D)$ cyclicity and $weak^\ast$ sequential cyclicity are equivalent notions. Therefore $S$ would be cyclic in $H^{2}(\D)$. But in $H^{2}(\D)$, the cyclic functions are precisely the outer functions, and a non-trivial singular inner function cannot be outer. This is a contradiction.

Thus, $f$ is cyclic in $H^{2}(\partial \mathbb{B}_{2})$, but it is not $weak^\ast$ sequentially cyclic.
\end{proof} 

Similarly, one obtains examples of this type  for any dimension $d$. Indeed, the singular inner function $S$ used above is cyclic in any weighted Bergman space of the unit disc with standard weights $(1-|z|^2)^\gamma dxdy$, $\gamma>-1$. Then we can define
\[
f(z_{1}, z_{2}, \dots, z_{d})
=
S(z_{1}).
\]
Repeating the above argument will show that $f$ is cyclic, but not $weak^\ast$ sequentially cyclic in $H^{2}(\partial \mathbb{B}_{d})$.

\section{Appendix}

Corollary \ref{cor:Pick} is based on Lemma 3.3 of \cite{AlemanHartzMcCarthyRichter2019}. Here we give a simple direct proof of the hypothesis for Lemma \ref{lem:weak*} for  Pick spaces.

\begin{proof} Assume that $H\subseteq {\mathrm{Hol}(\Bd)}$ has a normalized complete Pick kernel. Let $\psi \in \mathcal S(\phi)$, $\|\psi\|_{\MH}\le 1$. Then there are $u_n\in \MH$ such that $u_n\phi\to \psi$ $weak^\ast$ in $\MH$. Hence  $u_n\phi\to \psi$ weakly in $H$, and that implies that for some convex combination $v_n$ of the $u_n$'s we have $v_n\phi\to \psi $ in the norm of $H$. Then from the factorization theorem for Pick spaces, there are multipliers $a_n, b_n$ with $\|b_n\|_{\MH} \le 1$ and $\|a_n\|_{\MH} \le \|v_n\phi-\psi\|_H \to 0$ and $v_n\phi-\psi= \frac{a_n}{1-b_n}$, see \cite{AlemanHartzMcCarthyRichter2017}, Theorem 1.1. Now choose $r_n= 1-\sqrt{\|a_n\|_{\MH}}$.  Then
\begin{align*}
\|\frac{1-b_n}{1-r_nb_n} v_n \phi\|_{\MH}
&\le \| \frac{1-b_n}{1-r_nb_n}\psi\|_{\MH} +\| \frac{a_n}{1-r_nb_n}\|_{\MH}\\ 
& \le  \| \frac{1-b_n}{1-r_nb_n}\|_{\MH}\|\psi\|_{\MH} +   \| \frac{1}{1-r_n b_n}\|_{\MH} \| a_n\|_{\MH}\\
&\le \|\frac{1-z}{1-r_nz}\|_{H^\infty}+ \frac{\| a_n\|_{\MH}}{1-r_n}\\
&=\frac{2}{1+r_n}+\sqrt{\| a_n\|_{\MH}}\\
&\to 1.
 \end{align*} 
Here the last inequalities followed from von Neumann's inequality since 
$\|b_n\|_{\MH}=\|M_{b_n}\|_{\HB(H)}\le 1$. 

\end{proof}

\bibliographystyle{amsplain}  
\bibliography{Reference} 
\end{document}